%% file: main.tex
\documentclass{article}
\usepackage{hyperref}
\usepackage[numbers,square,comma]{natbib}
\input{packages}

\title{No cardinal correct inner model elementarily embeds into the universe}

\author{Gabriel Goldberg \\ Evans Hall \\
University Drive \\
Berkeley, CA 94720
\\ \vspace{0.5cm} \\ Sebastiano Thei\footnote{Sebastiano Thei was partially supported by the Italian PRIN 2017 Grant ``Mathematical Logic: models, sets, computability,'' by the Italian PRIN 2022 Grant ``Models, sets and classifications'' and by the European Union - Next Generation EU.}\\
Universit{\`a} degli Studi di Udine \\
Via delle Scienze, 206, 33100 Udine (UD), Italy \\
E-mail address:
thei91.seba@gmail.com}
\date{}

\begin{document}

\thispagestyle{plain}

\markboth{G. Goldberg \& S. Thei}{No cardinal correct inner model elementarily embeds into the universe}

\maketitle
\input{body}

\bibliographystyle{abbrv}
\bibliography{main}

\end{document}

%% file: packages.tex
\usepackage[latin1]{inputenc}
\usepackage{babel}
\usepackage{amsmath,amsfonts,amssymb,amsthm}
\usepackage{bbm}
\usepackage{hyperref}
\usepackage{verbatim}
\usepackage{tikz}
\usepackage{graphicx}

\usetikzlibrary{positioning}

\theoremstyle{plain}
\newtheorem{teo}{Theorem}[section]
\newtheorem{cor}[teo]{Corollary}
\newtheorem{prop}[teo]{Proposition}
\newtheorem{lem}[teo]{Lemma}
\newtheorem{cla}[teo]{Claim}

\newtheorem{con}[teo]{Conjecture}

\theoremstyle{defin}
\newtheorem{defin}[teo]{Definition}

\newtheorem{Q}[teo]{Question}

\newcommand{\system}[1]{\mbox{\fontfamily{cmss}\fontshape{n}\fontseries{m}\selectfont#1}}

\newcommand{\ZFC}{\system{ZFC}}

\newcommand{\GCH}{\system{GCH}}

\newcommand{\PFA}{\system{PFA}}
\newcommand{\MM}{\system{MM}}
\newcommand{\SSH}{\system{SSH}}
\newcommand{\BPFA}{\system{BPFA}}
\newcommand{\SCH}{\system{SCH}}

\newcommand{\subsetsim}{\mathrel{\ooalign{\raise.4ex\hbox{$\subset$}\cr$\raise-.9ex\hbox{$\sim$}$}}}

\DeclareMathOperator{\cof}{cof}

\DeclareMathOperator{\Ult}{Ult}
\DeclareMathOperator{\ot}{ot}

\DeclareMathOperator{\id}{id}

\DeclareMathOperator{\Card}{Card}

\DeclareMathOperator{\crit}{crit}
\DeclareMathOperator{\pp}{pp}
\DeclareMathOperator{\PP}{PP}

%% file: body.tex
\begin{center}
    \textbf{Abstract}
\end{center}
\begin{quote}
    An elementary embedding $j:M\rightarrow N$ between two inner
    models of $\ZFC$ is \emph{cardinal preserving}
if $M$ and $N$ correctly compute the class of cardinals. We look at the case $N=V$ and show that there is no nontrivial cardinal preserving elementary
embedding from $M$ into $V$, answering a question of Caicedo.
\end{quote}

\section{Introduction}

Large cardinal axioms are typically formulated in terms of elementary embeddings from the universe $V$ into some transitive subclass $M$. By demanding a stronger and stronger degree of resemblance
between $V$ and $M$, one obtains stronger
and stronger principles of infinity. For example, one may require $M$ to have more and more fragments of $V$: \emph{$\gamma$-strong} cardinals have $V_{\gamma}\subseteq M$, and \emph{$n$-superstrong} cardinals have $V_{j^{(n)}(\kappa)}\subseteq M$. Similarly, one may ask how close $M$ is: \emph{$\gamma$-supercompact} cardinals have $M^\gamma\subseteq M$, and \emph{$n$-huge} cardinals have $M^{j^{(n)}(\kappa)}\subseteq M$. 

Straining the limits of consistency, Reinhardt \cite{MR2617032} considered the natural extreme of this trend in which the target model is the entire universe. Few years later, Kunen \cite{MR0311478}, with his celebrated inconsistency theorem, refuted this suggestion and provided what seems to be an upper bound in the formulation of large cardinal axioms.
\begin{teo}[Kunen, \cite{MR0311478}]
    There is no nontrivial elementary embedding from the universe to itself.
\end{teo} Kunen's proof actually shows that if $j:V\rightarrow M$ is a nontrivial elementary embedding and $\lambda$ is the supremum of the \emph{critical sequence of $j$}, then $V_{\lambda+1}\nsubseteq M$. 
\begin{defin}
    Let $j : M \rightarrow N$ be an elementary embedding between two
transitive models of $\ZFC$. The \emph{critical sequence of $j$} is the sequence $\langle\kappa_n(j)\rangle_{n<\omega}$
defined by setting $\kappa_n(j) = j^{(n)}(\crit(j))$. The ordinal $\kappa_\omega(j)$ is the supremum of the
critical sequence of $j$.
\end{defin}
Looking for inconsistencies and with attention shifted upward to stronger principles, a new breed of
large cardinal hypothesis was introduced, the rank-into-rank embeddings. \emph{Axiom} $I_2(\lambda)$, for example, states that there is an elementary embedding $j:V\rightarrow M$ such that $\crit(j)<\lambda$, $j(\lambda)=\lambda$ and $V_\lambda\subseteq M$. Caicedo \cite{Ca} pushed this further and proposed another way to extend the large cardinal
hierarchy in $\ZFC$ and to obtain axioms just at the edge of Kunen inconsistency. Following the usual template of resembling $V$, the idea here is to impose agreement of cardinals between the models involved.
\begin{defin}
    Let $j : M \rightarrow N$ be an elementary embedding between two
transitive models of $\ZFC$. Then, $j$ is \emph{cardinal preserving} if $M$, $N$ and $V$ have the same class of cardinals, i.e. $\Card^M=\Card^N=\Card$.
\end{defin}
\begin{Q}[Caicedo]\label{Caicedo's question}
    Assume $j : M \rightarrow N$ is a nontrivial elementary embedding. Can $j$ be cardinal preserving?
\end{Q}
The expectation is that Question \ref{Caicedo's question} has a negative answer. Taking the first step
towards this line of research, Caicedo considered the case where either $M$ or $N$ is $V$. Both principles have significant consistency strength:
\begin{teo}[\cite{goldberg2021note}, Theorem 2.10]
    The existence of a cardinal preserving embedding $j:V\rightarrow N$ implies the consistency of $\ZFC$ $+$ there is a strongly compact cardinal.
\end{teo}
\begin{teo}[\cite{Ca}, Theorem 2.11]
    Assume that there is a cardinal preserving embedding $j : M \rightarrow V$. Then there are inner models with strong cardinals.
\end{teo}
In \cite[Corollary 2.10]{Ca}, it has been shown that $\PFA$ rules out the case $N=V$. On the other hand, the first author \cite[Theorem 6.6]{goldberg2020some} proved that the existence of a proper class of almost strongly compact cardinals refutes the case $M=V$. In this paper we show that cardinal preserving embeddings $j:M\rightarrow V$ are inconsistent with $\ZFC$. It is still open whether $\ZFC$ alone can refute cardinal preserving embeddings from $V$ to $N$.

It will be assumed throughout the paper that, if not otherwise specified, all embeddings are elementary and between transitive models of $\ZFC$.

\section{Forcing axioms} 
A first impulse for this investigation around cardinal preserving embeddings is due to the following conjecture \cite[Conjecture 1]{MR2385636}.

\begin{con}[Caicedo, Veli\v{c}kovi\'{c}]\label{Conj}
    Assume $W\subseteq V$ are models of $\MM$ with the same cardinals. Then $W$ and $V$ have the same $\omega_1$-sequences of ordinals. 
\end{con} 
The conjecture is motivated by a result of Caicedo and Veli\v{c}kovi\'{c} \cite[Theorem 1]{MR2231126} which asserts that, if $W\subseteq V$ are transitive models of $\ZFC+\BPFA$ and $\omega_2^W=\omega_2^V$, then $P(\omega_1)\subseteq W$. In light of this, one may ask whether any two models $W\subseteq V$ of some strong forcing axiom with
the same cardinals have the same $\omega_1$-sequences of ordinals. Thus, Caicedo and Veli\v{c}kovi\'{c}'s conjecture is a possible way to formalize this idea.

However, there is a tension between conjecture \ref{Conj} and cardinal preserving embeddings $j:M\rightarrow N$ with $M\vDash\MM$. Indeed, under $\MM$, the conjecture would refute these embeddings: 
\begin{prop}
    If $M\subseteq V$ are models of $\MM$ and $j:M\rightarrow N$ is a cardinal preserving embedding, then conjecture \ref{Conj} implies that $j$ is the identity map. 
\end{prop}
\begin{proof}
    Suppose towards a contradiction that $\kappa=\crit(j)$ exists and let $\lambda$ be $\kappa_\omega(j)$. Conjecture \ref{Conj} ensures that the sequence $\langle\kappa_n(j)\rangle_{n<\omega}$ belongs to $M$. Therefore, \begin{center}
$j(\lambda)=j(\sup_{n<\omega}\kappa_n(j))=\sup_{n<\omega}\kappa_{n+1}(j)=\lambda$.
\end{center}
    As  $\cof^M(\lambda)=\omega$, we may fix in $M$ a sequence cofinal in $\lambda$ consisting of successor cardinals $\Vec{\delta}=\langle\delta_n\rangle_{n<\omega}$, and a \emph{scale} $\langle f_\alpha\rangle_{\alpha<\lambda^+}$ at $\lambda$ relative to the sequence $\Vec{\delta}$. A scale here is a sequence of functions from $\prod_{n<\omega}\delta_n$ which is increasing and cofinal with respect to $<^\ast$. For functions $f$ and $g$ in $\prod_{n<\omega}\delta_n$, $f <^\ast g$
means that there exists $n< \omega$ such that for every $m>n$, $f(m) < g(m)$. Since $\MM$ implies $\SCH$ and $2^{\aleph_0}=\aleph_2$, it holds that $\lambda^\omega=\lambda^+$. Note that by closure under countable sequences and cardinal preservation, the statement $\lambda^\omega=\lambda^+$ is absolute between $M$, $N$ and $V$. So, using $\lambda^\omega=\lambda^+$, one can easily construct by induction such a sequence of $f_\alpha$'s. 

Without loss of generality, we may assume that each $\delta_n$ lives in the interval $(\kappa, \lambda)$. By elementarity, \begin{center}
        $N\vDash j(\langle f_\alpha\rangle_{\alpha<\lambda^+})=\langle g_\beta\rangle_{\beta<\lambda^+}$ is a scale at $\lambda$ relative to $j(\Vec{\delta})=\langle j(\delta_n)\rangle_{n<\omega}$.
    \end{center} But since $N$ is closed under countable sequences, being a scale is upwards absolute to $V$. As $j[\lambda^+]$ is cofinal in $\lambda^+$, $\langle g_\beta\rangle_{\beta\in j[\lambda^+]}$ is also a scale relative to $\langle j(\delta_n)\rangle_{n<\omega}$. Of course $g_{j(\alpha)} = j(f_\alpha)$ for all $\alpha<\lambda^+$, and so $\langle j(f_\alpha)\rangle_{\alpha<\lambda^+}$ is a scale relative to $\langle j(\delta_n)\rangle_{n<\omega}$ as well. Now we reach a contradiction following Zapeltal's proof of Kunen inconsistency \cite{MR1317054}. Let $h=\langle\sup j[\delta_n]\rangle_{n<\omega}$. Since $\lambda$ is the first fixed point above $\kappa$, $\delta_n<j(\delta_n)$. Moreover, $\delta_n$ is a successor cardinal, say $\delta_n=\nu_n^+$, and so $j(\delta_n)=(j(\nu_n)^+)^N=j(\nu_n)^+$. This means that $j(\delta_n)$ is a regular
cardinal larger than $\delta_n$, and so $\sup j[\delta_n] < j(\delta_n)$. Hence $h\in\prod_{n<\omega}j(\delta_n)$. Therefore there is $\alpha<\lambda^+$ such that $h<^\ast j(f_\alpha)$. But for all $n<\omega$, $j(f_\alpha)(n)=j(f_\alpha)(j(n))=j(f_\alpha(n))$ and since $f_\alpha(n)<\delta_n$, we have that $j(f_\alpha(n))\leq\sup j[\delta_n]=h(n)$. Therefore, for all $n$ big enough, we get \begin{center}
    $h(n)<j(f_\alpha)(n)=j(f_\alpha(n))\leq h(n)$,
\end{center} which is absurd.
\end{proof}
  All we used to get the contradiction in the proof above is the fact that the cardinal correct models $M$, $N$ and $V$ have the same $\omega$-sequences of ordinals, together with $\lambda^\omega=\lambda^+$. In \cite[Theorem 2.4]{7ae31891-508b-3b80-a3ca-95d64fa38807}, Foreman proved that if $j:M\rightarrow V$ is a nontrivial elementary embedding ($\Card^M\neq\Card$ is possible), then $^\omega M\nsubseteq M$. The argument above shows that this is also the case for embeddings of the form $j:M\rightarrow N$ that are cardinal preserving up to $\kappa_\omega(j)^+$. 
  \begin{prop}\label{cardinal correctness and countable sequences}
      Let $j:M\rightarrow N$ be a nontrivial elementary embedding between two transitive models of $\ZFC$, and let $\lambda=\kappa_\omega(j)$. Suppose $\lambda^\omega=\lambda^+$ and $\Card^M\cap\lambda^+=\Card^N\cap\lambda^+=\Card\cap\lambda^+$. Then either $j(\lambda)>\lambda$ or ${}^\omega\lambda\nsubseteq N$. In particular, ${}^\omega\lambda\nsubseteq M\cap N$.
  \end{prop}
  The first author observed in \cite{goldberg2020some} that $N$ cannot be closed under $\omega$-sequences, whenever $N$ is the target model of a cardinal preserving elementary embedding whose domain is $V$. Exploiting the following result by Viale one can show that, under $\PFA$, the same conclusion holds in a more general case.  
  \begin{lem}[\cite{MR2385636}, Corollary 28]\label{PFA and cardinal correctness}
      If $\PFA$ holds and $M$ is an inner model with the same cardinals, then $M$
computes correctly all ordinals of cofinality $\omega$.
  \end{lem}
  \begin{cor}[$\PFA$]
      Let $j:M\rightarrow N$ be a cardinal preserving embedding. Then ${}^\omega\lambda\nsubseteq N$, where $\lambda=\kappa_\omega(j)$.
  \end{cor}
  \begin{proof}
      By Lemma \ref{PFA and cardinal correctness}, $\cof^M(\lambda)=\omega$ and so $j(\lambda)=\lambda$. Hence ${}^\omega\lambda\nsubseteq N$, by Proposition \ref{cardinal correctness and countable sequences}.
  \end{proof}
\section{Singular cardinal combinatorics}

The proof of nonexistence of $j:M\rightarrow V$ with $\Card^M=\Card$ involves singular cardinal combinatorics. More specifically, it relies on some results concerning square principles (Magidor-Sinapova \cite{MR3692010}), good scales, J\'{o}nsson cardinals (Shelah \cite{Sh:g}), $\omega_1$-strongly compact cardinals (Bagaria-Magidor \cite{a4c6f821-6ac9-3a48-a9eb-4ca5d88ea129}) and basic facts from Shelah's pcf theory. 

Accordingly, with a view to the proof of Theorem \ref{main thm} below, we collect some definitions and results, with the further intent of fixing the notation. In the following, unless stated otherwise, all embeddings we will consider are supposed to be nontrivial.
\paragraph{Weak forms of square.} The notion of square principle, denoted $\Box_\kappa$, is a central concept in Jensen's \cite{JENSEN1972229} fine structure analysis
of $L$. For a cardinal $\kappa$, $\Box_\kappa$ states that there exists a sequence $\langle C_\alpha:\alpha<\kappa^+\rangle$
such that each $C_\alpha$ is a club subset of $\alpha$, $\ot(C_\alpha)\leq\kappa$, and if $\delta\in\lim C_\alpha$,
then $C_\alpha \cap \delta = C_\delta$. In his study of 
core models for Woodin 
cardinals, Schimmerling \cite{SCHIMMERLING1995153} isolated a spectrum of square principles $\Box_{\lambda,\kappa}$, for $1\leq\lambda\leq\kappa$, 
 that form a natural hierarchy below $\Box_\kappa$. We are interested in a specific weakening of $\Box_\kappa$ that allows at most countably many guesses for the clubs at each point.
\begin{defin}
    Let $\kappa$ be a cardinal. The principle $\Box_{\kappa,\omega}$ asserts that there exists a sequence $\langle\mathcal{C}_\alpha:\alpha<\kappa^+\rangle$ such that, for all $\alpha<\kappa^+$, \begin{enumerate}
    \item $1\leq |\mathcal{C}_\alpha|\leq\omega$, 
    \item every $C\in\mathcal{C}_\alpha$ is a club subset of $\alpha$ with $\ot(C)\leq\kappa$, and
    \item if $C\in\mathcal{C}_\alpha$ and $\delta\in\lim(C)$, then $C\cap\delta\in\mathcal{C}_\delta$.
\end{enumerate}
\end{defin}
Cummings and Schimmerling \cite{MR1942302} showed that after Prikry forcing at $\kappa$, $\Box_{\kappa,\omega}$ holds in the generic extension. Magidor and Sinapova \cite{MR3692010} observed that arguments in Gitik \cite{MR1469092} and independently in D\'{z}amonja-Shelah \cite{MR1360144} yield a more general result for $\Box_{\kappa,\omega}$:
\begin{lem}\label{weak square in the outer model}
     Let $V\subseteq W$ be transitive class models of $\ZFC$. Suppose that $\kappa$ is an inaccessible cardinal in $V$, singular of countable cofinality in $W$,
and $(\kappa^+)^V=(\kappa^+)^W$. Then $W\vDash\Box_{\kappa,\omega}$.
\end{lem}
There is a connection between this square principle and a pcf-theoretic object called \emph{good scale}.
\begin{defin}\label{scale, good scale}
    Suppose $\langle\kappa_n\rangle_{n<\omega}$ is an increasing sequence of regular cardinals with $\sup_{n<\omega}\kappa_n=\lambda$. Let $\Vec{f}=\langle f_\alpha\rangle_{\alpha<\lambda^+}$ be a sequence of functions from $\prod_{n<\omega}\kappa_n$. The sequence $\Vec{f}$ is a \emph{good scale at} $\lambda$ if it is a scale relative to the sequence $\langle\kappa_n\rangle_{n<\omega}$, and for all $\nu<\lambda^+$ with $\cof(\nu)>\omega$, there is a cofinal subset $C$ of $\nu$ such that, for some $n<\omega$, and for all $k\geq n$, the sequence $\langle f_\alpha(k):\alpha\in C\rangle$ is strictly increasing.
\end{defin}

\begin{lem}[\cite{Cummings2001-CUMSSA}, Theorem 3.1]\label{good scales with square}
    Let $\lambda$ be a singular cardinal. Then $\Box_{\lambda,\omega}$ implies that there is a good scale at $\lambda$.
\end{lem}

\paragraph{J\'{o}nsson cardinals.} Another crucial role for our argument will be played by J\'{o}nsson cardinals, and their influence on the pcf structure: the presence of successor J\'{o}nsson cardinals implies the failure of $\SSH$, which in turn ensures the existence of good scales.
\begin{defin}
\begin{enumerate}
    \item Suppose that $\lambda$ is a singular cardinal, and let $\PP(\lambda)$ be
the set of all cardinals of the form $\cof(\prod A/D)$, where $A$ is a set of regular
cardinals cofinal in $\lambda$ of order-type $\cof(\lambda)$ and $D$ is an ultrafilter on $A$ containing no bounded subsets. We define the \emph{pseudopower of} $\lambda$, denoted $\pp(\lambda)$, as the supremum of $\PP(\lambda)$. In symbols, \begin{center}
    $\pp(\lambda)=\sup\PP(\lambda)$.
\end{center}
\item  \emph{Shelah's Strong Hypothesis}, denoted by $\SSH$, states that $\pp(\lambda) = \lambda^+$ for every singular cardinal $\lambda$.
\end{enumerate}
\end{defin}
\begin{defin}
    A cardinal $\kappa$ is called \emph{J\'{o}nsson} if every first-order structure of cardinality $\kappa$ 
whose language is countable possesses a \emph{J\'{o}nsson substructure}, i.e. a proper 
elementary substructure of the same cardinality. 
\end{defin}
The following characterization of J\'{o}nsson cardinals in terms of elementary embeddings will be useful.
\begin{prop}[\cite{MR0788256}, Lemma 1]\label{Jonnsoness}
    Let $\kappa$ be an  uncountable cardinal. Then the following are equivalent:
    \begin{enumerate}
        \item $\kappa$ is J\'{o}nsson.
        \item For some $\alpha>\kappa$, there is an elementary embedding $j:N\rightarrow V_\alpha$ such that $N$ is a transitive set, $j(\kappa)=\kappa$ and $\crit(j)<\kappa$.
    \end{enumerate}
\end{prop}
Erd\H{o}s and Hajnal \cite{MR0209161} proved that, under $\GCH$, a successor cardinal cannot be
J\'{o}nsson. Tryba showed in $\ZFC$ alone a special case of this result.
\begin{teo}[\cite{MR0788256}, Theorem 2]\label{succ of reg not Jon}
    If a regular cardinal $\kappa$ is J\'{o}nsson, then every
stationary $S \subseteq \kappa$ reflects.
    In particular, if $\kappa$ is a regular cardinal, then $\kappa^+$ is not J\'{o}nsson.
\end{teo}
Proposition \ref{Jonnsoness} and Theorem \ref{succ of reg not Jon} provide a new proof of the result by Kunen quoted in the Introduction.
\begin{teo}[Kunen]
    There is no elementary embedding from the universe to itself.
\end{teo}
\begin{proof}
    Suppose $j:V\rightarrow V$ is an elementary embedding. Then $\kappa_{\omega}(j)^{++}$ is fixed by $j$. Let $\alpha$ be a fixed point of $j$ above $\kappa_{\omega}(j)^{++}$, and consider the elementary embedding $j\restriction V_\alpha:V_\alpha\rightarrow V_\alpha$. By Proposition \ref{Jonnsoness}, $\kappa_{\omega}(j)^{++}$ is J\'{o}nsson. On the other hand, $\kappa_{\omega}(j)^{++}$ is the successor of a regular cardinal and so, by Theorem \ref{succ of reg not Jon}, it cannot be J\'{o}nsson.
\end{proof}
Among other things, Theorem \ref{succ of reg not Jon} says that, if $\lambda^+$ is J\'{o}nsson, $\lambda$ has to be singular and $\GCH$ fails. Therefore it makes sense to ask whether $\SSH$ holds at $\lambda$. 

The following lemma is due to Shelah. The first assertion can be found in \cite[Corollary 5.9]{Eisworth2010}, while the second one is proved in \cite[Theorem 4.78]{Eisworth2010} (see also \cite[Chapter 2, Claim 1.3]{Sh:g}).
\begin{lem}\label{good scales from not SSH}
    Let $\lambda$ be a singular cardinal. \begin{enumerate}
        \item If $\lambda^+$ is J\'{o}nsson then $\pp(\lambda)>\lambda^+$.
        \item If $\pp(\lambda)>\lambda^+$ then $\lambda$ carries a good scale.
    \end{enumerate}
\end{lem}

\vspace{1cm}
\section{Discontinuities}

Using the alternative definition of J\'{o}nsson cardinal (Proposition \ref{Jonnsoness}), we provide a different proof of the following lemma due to Caicedo.

\begin{lem}[\cite{Ca}, Theorem 2.5]
\label{no fixed points} 
If $j:M\rightarrow V$ is a cardinal preserving embedding, then $j$ has no fixed points above its critical point. In particular, for all $\lambda>\crit(j)$, if $j[\lambda]\subseteq\lambda$, then $\cof^M(\lambda)\geq\crit(j)$.
\end{lem}
\begin{proof}
    Suppose towards a contradiction that there is some ordinal $\alpha$ above $\crit(j)$ such that $j(\alpha)=\alpha$. By cardinal correctness, $j(\alpha^{++})=\alpha^{++}$. Pick an ordinal $\beta>\alpha^{++}$ and stipulate $N=(V_\beta)^M$. Then $j\restriction N$ is an elementary embedding from $N$ to $V_{j(\beta)}$. To see this, note that for $b$ in $N$ and $\psi$ a first-order formula, $N\vDash\psi(b)$ is equivalent to $M\vDash\psi^N(b)$. By elementarity, $M\vDash\psi^N(b)$ if and only if $V\vDash\psi^{V_{j(\beta)}}(j(b))$. Finally, $V\vDash\psi^{V_{j(\beta)}}(j(b))$ is equivalent to $V_{j(\beta)}\vDash\psi(j(b))$. Therefore $j\restriction N$ witnesses that $\alpha^{++}$ is J\'{o}nsson, contradicting the fact that the successor of a regular cardinal cannot be J\'{o}nnson (Theorem \ref{succ of reg not Jon}). 

    To prove the last assertion suppose $\lambda>\crit(j)$, $j[\lambda]\subseteq\lambda$ and $\cof^M(\lambda)<\crit(j)$. Let $A$ be a cofinal subset of $\lambda$ in $M$ with $|A|<\crit(j)$. By elementarity, $j(A)$ is a cofinal subset of $j(\lambda)$. Hence \begin{center}
        $j(\lambda)=\sup j(A)=\sup j[A]\leq\sup j[\lambda]\leq\sup\lambda=\lambda$,
    \end{center} contradicting the fact that $\lambda<j(\lambda)$.
\end{proof}
An immediate corollary is that, if $j:M\rightarrow V$ is a cardinal preserving embedding, then $M$ cannot closed under sequences of size less than $\crit(j)$.

 Despite this discontinuity, we still have a proper class of \emph{closure points}.
\begin{defin}
    Let $j:M\rightarrow N$ be an elementary embedding between two transitive models of $\ZFC$. A cardinal $\lambda$ is \emph{closure point of} $j$ if $j[\lambda]\subseteq\lambda$. 
\end{defin}
So for example if $j:M\rightarrow V$ is cardinal preserving, then $\kappa_\omega(j)$ is a closure point. More generally, one can easily find an $\omega$-club class consisting of strong limit closure points of $j$ of countable cofinality above  $\crit(j)$. 
We will show that such cardinals are either $M$-regular cardinals or predecessors of a J\'{o}nsson cardinal.

\section{The main result}
\paragraph{Factor embeddings.} The key step in the proof of Lemma \ref{regular or pred of Jonsson} below deals with ultrafilters derived from an embedding, and  applied to a model to which they do not belong. Accordingly, we review some basic definitions and facts about relativized ultrapowers and factor embeddings.
\begin{defin}
   Let $j : N \rightarrow P$ be an
elementary embedding between two transitive models of $\ZFC$ and let $x$ be a set in $N$. Suppose $a \in j(x)$. \begin{itemize} 
        \item An $N$\emph{-ultrafilter on} $x$ is a set $U\subseteq P(x)\cap N$ such that
   \begin{center}
       $(N, U)\vDash U$ is an ultrafilter. 
       \end{center}
\item The $N$\emph{-ultrafilter on} $x$ \emph{derived from} $j$ \emph{using} $a$ is the $N$-ultrafilter 
\begin{center}
    $\{A \in P(x) \cap N : a \in j(A)\}$.
    \end{center}
    \end{itemize}
\end{defin} 
If $U$ is an $N$-ultrafilter on a set $x$, $M_U^N$ denotes the unique transitive collapse of the class $\Ult(N,U)=\{[f]_U^N: f:x\rightarrow N\}$, where $[f]_U^N$ is the set of functions $g:x\rightarrow N$ in $N$ such that $g=_U f$ and for each function $h:x\rightarrow N$ in $N$, if $f=_U h$ then $\mathrm{rank}(g)\leq\mathrm{rank}(h)$. The latter requirement ensures that $[f]_U^N$ is a set and it is known as Scott's trick. $\Ult(N,U)$ is called \emph{relativized ultrapower of} $N$ \emph{by} $U$, and will be tacitly identified with $M_U^N$. $N$ is elementarily embeddable in its ultrapower via the map $j^N_U:N\rightarrow M_U^N$, defined as \begin{center}
    $j^N_U:y\mapsto [c_y]_U^N$.
\end{center} We refer to $j^N_U$ as the \emph{the canonical embedding of} $N$ \emph{in} $M^N_U$. 

The following is an useful relationship between an elementary embedding and the ultrapowers associated to its derived ultrafilters.
\begin{lem}\label{factor embedding}
    Let $j : N \rightarrow P$ be an
elementary embedding between two transitive models of $\ZFC$. Suppose $x\in N$ and $a \in j(x)$. Let $U$ be the $N$-ultrafilter on $x$
derived from $j$ using $a$. Then there is an elementary embedding $k:M_U^N\rightarrow P$ such that $k\circ j_U^N=j$ and $k([\id]_U^N)=a$.
\end{lem}
\begin{proof}
    For each $[f]_U^N\in M^N_U$, stipulate $k([f]_U^N)=j(f)(a)$. It is routine to verify that $k$ fulfills the desired properties.
\end{proof}
We refer to the embedding $k$ as the \emph{factor embedding} associated to the derived $N$-ultrafilter $U$.

\paragraph{Good scales at $\lambda$.} In the following $j:M\rightarrow V$ is a cardinal preserving elementary embedding and $\lambda$ is a strong limit closure point of $j$ of countable cofinality above $\crit(j)$.
\begin{lem}\label{regular or pred of Jonsson}
    Either $\lambda$ is regular in $M$ or $\lambda^+$ is J\'{o}nsson.
\end{lem}
\begin{proof}
    Suppose $\lambda$ is singular in $M$. First we factor the embedding $j$. Let $D$ be the $M$-ultrafilter on $\lambda$ derived from $j$ using $\lambda$, let $j_D^M:M\rightarrow M_D^M$ be the canonical embedding of $M$ in $M^M_D$, and let $k : M^M_D \rightarrow V$ be the factor embedding associated to $D$. An easy argument yields $[\id]^M_D=\lambda$. Nevertheless, we provide the proof just to clarify where the closure of $\lambda$ is used.
    \begin{cla}
        $[\id]^M_D=\lambda$.
    \end{cla}
    \begin{proof}[Proof of claim]
        Let $\alpha<\lambda$. As $j[\lambda]\subseteq\lambda$, $j(\alpha)<\lambda$. So $\lambda$ belongs to the set 
            $\{\beta<j(\lambda):j(\alpha)<\beta\}$.
     By definition of $D$, $\{\beta<\lambda:\alpha<\beta\}\in D$. \L o\'{s}'s theorem yields $\alpha\leq j_D^M(\alpha)=[c_\alpha]_D^M<[\id]_D^M$. Now suppose $\alpha<[\id]_D^M$, say $\alpha=[f]_D^M$ for some function $f:\lambda\rightarrow\lambda$ in $M$. By \L o\'{s}'s theorem, $\{\beta<\lambda:f(\beta)<\beta\}\in D$. Hence $\lambda\in\{\beta<j(\lambda):j(f)(\beta)<\beta\}$. Finally, $\alpha=[f]^M_D\leq k([f]^M_D)=j(f)(\lambda)<\lambda$.  
    \end{proof}
    Therefore $\lambda$ is fixed by $k$. A standard argument shows that $\lambda^+$ is correctly computed by $M_D^M$, that is $(\lambda^+)^{M_D^M}=\lambda^+$. To see this, pick a wellorder $\prec$ of $\lambda$ in $M$. By elementarity, $\triangleleft=j_D^M(\prec)\cap(\lambda\times\lambda)$ is a wellorder of $\lambda$ in $M^M_D$ with length $\geq$ $\ot(\prec)$. Thus $(\lambda^+)^{M^M_D}>\ot(\triangleleft)\geq\ot(\prec)$. If $(\lambda^+)^{M^M_D}<(\lambda^+)^M$, there are in $M$ a bijection $f:\lambda\rightarrow(\lambda^+)^{M_D^M}$ and a wellorder $\prec$ of $\lambda$ given by $\alpha\prec\beta$ iff $f(\alpha)\in f(\beta)$, leading to the following contradiction: \begin{center}   $(\lambda^+)^{M_D^M}>\ot(j^M_D(\prec)\cap(\lambda\times\lambda))\geq\ot(\prec)=(\lambda^+)^{M_D^M}$.
\end{center} 
So $(\lambda^+)^{M^M_D}$ has to be $(\lambda^+)^M$. In particular, $(\lambda^+)^{M_D^M}=(\lambda^+)^M=\lambda^+$ and $k(\lambda^+)=\lambda^+$. 
\begin{cla}
    $\crit(k)<\lambda$.
\end{cla}
\begin{proof}[Proof of claim]
    Note that since $\lambda$ is singular in $M$, there is a set $C\in D$ such that $|C|^M < \lambda$. In fact, let $C\in M$ be a closed unbounded subset of $\lambda$
    of ordertype $\cof^M(\lambda)$.
    Then \(j(C)\) is closed in $j(\lambda) > \lambda$, and $j(C)\cap\lambda$ is unbounded in $\lambda$ since it contains $j[C].$ Therefore $\lambda\in j(C),$ so $C\in D$ by the definition of a derived ultrafilter.

    Now $j_D^M$ is continuous at every $M$-regular $\gamma < \lambda$ such that $\gamma > |C|^M$. On the other hand, $j$ is discontinuous at every $M$-regular cardinal by Lemma \ref{no fixed points}.
    Therefore letting $\gamma < \lambda$ be any $M$-regular cardinal such that $\gamma > |C|^M$, we have that $k(j_D^M(\gamma)) = j(\gamma) > j_D^M(\gamma)$, and so since $j_D^M(\gamma) < \lambda$, $\crit(k) < \lambda$.
\end{proof}
    As argued in the proof of Lemma \ref{no fixed points}, we get that $\lambda^+$ is J\'{o}nsson: pick $\beta>\lambda^+$ and let $N$ be $(V_\beta)^{M_D^M}$. Then $k\restriction N:N\rightarrow V_{k(\beta)}$ is an elementary embedding witnessing the J\'{o}nssonness of $\lambda^+$.
\end{proof}
We have already pointed out that both cases lead to some kind of incompactness. Suppose $\lambda$ is regular in $M$. Then $\lambda$, being strong limit, is inaccessible in $M$. So we can apply Lemma \ref{weak square in the outer model} to infer that $\Box_{\lambda,\omega}$ holds. Moreover, Lemma \ref{good scales with square} ensures that $\lambda$ carries a good scale. On the other hand, if $\lambda^+$ is J\'{o}nsson, we get the same conclusion by Lemma \ref{good scales from not SSH}. Altogether we deduce a more quotable corollary:
\begin{cor}\label{good scales yes}
    There is a good scale at $\lambda$.
\end{cor}
Now we aim to show that this cannot be the case.
\paragraph{Embeddings into $V$.} In \cite{a4c6f821-6ac9-3a48-a9eb-4ca5d88ea129}, Bagaria and Magidor proved that $\SCH$ holds above an $\omega_1$-strongly compact cardinal. They essentially used the following theorem, together with a result due to Shelah \cite{Sh:g} asserting that the failure of $\SCH$ implies the existence of a good scale.

\begin{teo}[Bagaria-Magidor]\label{Bag-Mag thm} If $\kappa$ is $\omega_1$-strongly compact, then there is no $\lambda>\kappa$ of countable cofinality carrying a good scale.
\end{teo}
We will not need Bagaria and Magidor's theorem but rather its proof. Indeed, it carries over exactly to our context:
\begin{lem}\label{no good scales}
    If $j:M\rightarrow V$ is cardinal preserving, then no singular cardinal $\lambda>j(\crit(j))$ of countable cofinality carries a good scale.
\end{lem}
\begin{proof}
    Towards a contradiction, suppose there is a singular cardinal greater than $j(\crit(j))$ of countable cofinality carrying a good scale. By elementarity, this is true in $M$, namely \begin{center}
        $M\vDash$ there is a cardinal $\lambda>\crit(j)$ with $\cof(\lambda)=\omega$ carrying a good scale.
    \end{center}
    
    Working in $M$, let $\langle f_\alpha\rangle_{\alpha<\lambda^+}$ be a good scale, relative to an increasing sequence of regular cardinals $\langle\lambda_n\rangle_{n<\omega}$ with limit $\lambda$. By Lemma \ref{no fixed points} $\lambda^+<j(\lambda^+)$. Passing to $V$, let $\beta=\sup j[\lambda^+]$. Since $\lambda^{+M}=\lambda^+$, we have $j(\lambda^+)=j(\lambda^{+M})=j(\lambda)^+$. In particular, $j(\lambda^+)$ is regular. On the other hand, $\beta$ is the supremum of a subset of $j(\lambda^+)$ of cardinality
 less than $j(\lambda^+)$, and therefore bounded in $j(\lambda^+)$. Thus, $\beta<j(\lambda^+)$. Elementarity of $j$ leads to \begin{center}
     $V\vDash j(\langle f_\alpha\rangle_{\alpha<\lambda^+})$ is a good scale relative to $\langle j(\lambda_n)\rangle_{n<\omega}$.
 \end{center}
Say $j(\langle f_\alpha\rangle_{\alpha<\lambda^+})=\langle g_\alpha\rangle_{\alpha<j(\lambda^+)}$. Since $\beta<j(\lambda^+)$ and has uncountable cofinality, we can use the goodness of $\langle g_\alpha\rangle_{\alpha<j(\lambda^+)}$ to pick a cofinal subset $C$ of $\beta$ such that, for some $n<\omega$, for all $\xi_0<\xi_1$ in $C$, and for all $k\geq n$, $g_{\xi_0}(k)<g_{\xi_1}(k)$. Following the proof of \cite[Theorem 4.1]{a4c6f821-6ac9-3a48-a9eb-4ca5d88ea129}, we will define by induction on $\delta<\lambda^+$ a strictly increasing sequence of ordinals $\langle \gamma_\delta\rangle_{\delta<\lambda^+}$ contained in $C$, and an auxiliary sequence of ordinals $\langle\alpha_\delta\rangle_{\delta<\lambda^+}$ such that $\gamma_\delta<j(\alpha_\delta)<\gamma_{\delta+1}$, for all $\delta<\lambda^+$. Let $\gamma_0$ be the first ordinal in $C$. Let $\alpha_0$ be the least ordinal such that $\gamma_0 < j(\alpha_0)$.
 Then let $\gamma_1\in C$  be such that $j(\alpha_0) < \gamma_1$. Then, let $\alpha_1$ be the least ordinal such that
 $\gamma_1 < j(\alpha_1)$. And so on. At limit stages, take the least $\gamma\in C$ greater than all the
 ordinals $\gamma_\delta$ picked so far. Clearly, $\alpha_\delta<\lambda^+$, for all $\delta<\lambda^+$. For each $\delta<\lambda^+$, we have $g_{\gamma_\delta}<^\ast g_{j(\alpha_\delta)}<^\ast g_{\gamma_{\delta+1}}$, and so we may pick some $n_\delta>n$ such that $g_{\gamma_\delta}(m)<g_{j(\alpha_\delta)}(m)<g_{\gamma_{\delta+1}}(m)$, for all $m\geq n_\delta$. By the pigeonhole principle, there is some $D\subseteq\lambda^+$ of cardinality $\lambda^+$ such that for all $\delta\in D$, the $n_\delta$ is the same, say $k$. In particular, if $\delta\in D$, then $g_{\gamma_\delta}(k)<g_{j(\alpha_\delta)}(k)<g_{\gamma_{\delta+1}}(k)$. On the other hand, if $\delta_0,\delta_1\in D$ and $\delta_0+1<\delta_1$, then $g_{\gamma_{\delta_0+1}}(k)<g_{\gamma_{\delta_1}}(k)$, by goodness. Therefore, \begin{center}
     $g_{\gamma_{\delta_0}}(k)<g_{j(\alpha_{\delta_0})}(k)<g_{\gamma_{\delta_0+1}}(k)<g_{\gamma_{\delta_1}}(k)<g_{j(\alpha_{\delta_1})}(k)<g_{\gamma_{\delta_1+1}}(k)$,
 \end{center}
    whenever $\delta_0,\delta_1\in D$ and $\delta_0+1<\delta_1$. Note that for every $\delta<\lambda^+$, \begin{center}
        $g_{j(\alpha_\delta)}(k)=j(f_{\alpha_\delta}(k))\in j[\lambda_k]$.
    \end{center}
    But this is impossible since the sequence $\langle g_{\gamma_\delta}(k)\rangle_{\delta\in D}$ has ordertype $\lambda^+$, and $\ot(\langle g_{\gamma_\delta}(k)\rangle_{\delta\in D})=\ot(\langle g_{j(\alpha_\delta)}(k)\rangle_{\delta\in D})\leq \ot(j[\lambda_k])$. However, the ordertype of $j[\lambda_k]$ is $\lambda_k<\lambda^+$.  This is a contradiction.
\end{proof}
\begin{teo}\label{main thm}
    There are no nontrivial cardinal preserving elementary embeddings from an inner model $M$ into the universe of sets $V$.
\end{teo}
\begin{proof}
    Suppose not and let $j:M\rightarrow V$ be a cardinal preserving elementary embedding. Let $\lambda$ be the first strong limit closure point of $j$ of countable cofinality strictly above $j(\crit(j))$. By Corollary \ref{good scales yes}, there is good scale at $\lambda$. But Lemma \ref{no good scales} says that this is impossible. 
\end{proof}